\documentclass[11pt]{amsart}
\usepackage{epic,eepic}
\usepackage{amscd,amsfonts,amsmath,amsxtra,amssymb}
\usepackage[all]{xy}
\usepackage{a4wide}

\begin{document}


\newcommand{\proj}{\pi}
\newcommand{\deck}{\varphi}
\newcommand{\C}{\mathbb C}
\newcommand{\R}{\mathbb R}
\newcommand{\PP}{\mathbb P}
\newcommand{\ZZ}{\mathbb Z}
\newcommand{\NN}{\mathbb N}
\newcommand{\lra}{\longrightarrow}
\newcommand{\CC}{{\mathbb C}}
\newcommand{\RR}{{\mathbb R}}
\def\OO{{\mathcal O}}
\def\caln{{\mathcal N}}
\newcommand\newop[2]{\def#1{\mathop{\rm #2}\nolimits}}
\newop\Big{Big}
\newop\vol{vol}
\newop\NS{NS}
\newop\supp{supp}

\title{When are Zariski chambers numerically determined?}
\author {S{\l}awomir Rams}
\thanks{Partial funding by NCN grant no. N N201 608040 is gratefully acknowledged (S. Rams).}
 \address{Institute of Mathematics, Jagiellonian University \\
 ul.~{\L}ojasiewicza~6 \\ 30-348 Krak\'ow \\ POLAND}
 \address{Institut fuer Algebraische Geometrie, Leibniz Universit\"at \\
Welfengarten~1 \\ 30167
 Hannover \\ GERMANY}
\email{slawomir.rams@uj.edu.pl and rams@math.uni-hannover.de}
\author{Tomasz Szemberg}
\thanks{Partially supported by NCN grant UMO-2011/01/B/ST1/04875 (T. Szemberg).}
\address{Instytut Matematyki UP\\
   ul. Podchor\c a\.zych 2\\
   PL-30-084 Krak\'ow, Poland}
\email{tomasz.szemberg@gmail.com}


\begin{abstract}
   The big cone of every smooth projective surface $X$ admits the natural
   decomposition into Zariski chambers. The purpose of this note is to give a simple criterion for
   the interiors of all Zariski chambers on $X$ to be numerically determined Weyl
   chambers. Such a criterion generalizes the results of Bauer-Funke \cite{bf} on K3 surfaces to arbitrary smooth projective surfaces.
   In the last section, we study the relation between decompositions of the big cone and elliptic fibrations  on Enriques surfaces.
\end{abstract}

\maketitle

\newcommand{\D}{\displaystyle}
\newcommand{\vv}{\vspace{0.3cm}}
\newcommand{\vV}{\vspace{0.5cm}}
\newcommand{\VV}{\vspace{2cm}}

\theoremstyle{remark}
\newtheorem{obs}{Observation}
\newtheorem{rem}[obs]{Remark}

\theoremstyle{definition}
\newtheorem{defi}[obs]{Definition}
\newtheorem{definition}[obs]{Definition}
\newtheorem{remark}[obs]{Remark}
\newtheorem{example}[obs]{Example}
\newtheorem{prop}[obs]{Proposition}
\newtheorem{theo}[obs]{Theorem}
\newtheorem{theorem}[obs]{Theorem}
\newtheorem{lemma}[obs]{Lemma}
\newtheorem{cor}[obs]{Corollary}
\newtheorem*{THM}{Theorem}
\newcommand{\ux}{\underline{x}}

\section{Introduction}
The main purpose of this note is to study numerical properties
of the decomposition of the big cone of smooth projective surfaces into Zariski chambers, i.e. the decomposition
induced by the variation of the Zariski decomposition of divisors over the big cone.
Recall that, given a pseudo-effective $\R$-divisor $D$ on a smooth projective surface $X$,
there exist effective $\R$-divisors $P_D$ and $N_D$ such that
\begin{equation} \label{eq-zariski}
D=P_D+N_D
\end{equation}
and the following conditions are satisfied

\begin{itemize}
\item[(Z1)]  the divisor $P_D$ is nef;

\item[(Z2)] either $N_D=0$ or  $N_D=\sum_{i=1}^{s}\alpha_iC_i$ and the intersection matrix
$[C_i.C_j]_{i,j=1,\ldots,n}$ is negative-definite;

\item[(Z3)] one has $P_D.C_i=0$ for all $i=1,\ldots,s$.
\end{itemize}
\noindent
The divisor $P_D$ (resp. $N_D$) in \eqref{eq-zariski} is called the \emph{positive}
(resp. the \emph{negative}) part of $D$.
One can show (see \cite{zariski} or \cite{bauer} for a short proof in modern language) that
the \emph{Zariski decomposition} \eqref{eq-zariski} of the divisor $D$ is uniquely determined by the conditions (Z1), (Z2) and (Z3).
Moreover, all sections of $D$ come in effect from $P_D$, which can be expressed in terms
of the \emph{volume} $\vol(D)=\vol(P_D)$, (see \cite{PAG} for details).

Given an algebraic surface $X$, by  \cite[Theorem~1.2]{bks}, the variation of the Zariski decomposition over the big cone $\Big(X)$
leads to the \emph{Zariski  decomposition}  of  the cone $\Big(X)$. Indeed,
suppose that $P$ is a big and nef divisor. Recall the following definition (see \cite[p.~214]{bks}).

\begin{definition}[Zariski chamber] \label{def-zchamber}
The \emph{Zariski chamber} $\sum_P$ associated to $P$ is defined as
\begin{eqnarray*}
\textstyle{\sum_P} &:=&  \{ B \in \mbox{Big}(X) \, : \,  \mbox{irreducible components of }  N_B \mbox{ are the only irreducible curves} \\
                   &  &   \quad \mbox{ on } X \mbox{ that intersect  } P \mbox{ with multiplicity  } 0 \, \}.
\end{eqnarray*}
\end{definition}

\noindent
By \cite[Theorem~1.2]{bks}
Zariski chambers yield a locally finite decomposition of the 
cone $\Big(X)$  
   into locally polyhedral subcones such that the support of the
   negative part of the Zariski decomposition of all divisors in the subcone is constant.

On the other hand, it follows immediately from the property (Z2)
(see \eqref{eq-zariski}),   that the negative part $N_D$ of the Zariski decomposition is either
   trivial or its support consists  of \emph{negative} curves, i.e. curves with negative self-intersection.
   One can use such curves to define another decomposition of the big cone.
 Let $\caln(X)$ be the set of all \emph{irreducible} negative curves on $X$.
   Each curve $C\in\caln(X)$ defines the hyperplane in the N\'eron-Severi space $\NS_{\R}(X)$ of $X$
   $$C^{\perp}=\left\{D:\;D.C=0\right\}\subset\NS_{\R}(X),$$
   and the decomposition of the set
\begin{equation} \label{eq-bmh}
\Big(X) \setminus \bigcup\limits_{C\in\caln(X)}C^{\perp}
\end{equation}
into connected components yields a  decomposition
   of (an  open and dense subset of) the cone $\Big(X)$ into subcones.

\begin{definition}[simple Weyl chamber] Connected components of the set \eqref{eq-bmh}  are called
  \emph{simple Weyl chambers} of $X$.
\end{definition}
   Traditionally the (simple) Weyl chambers are defined if $X$ is a surface carrying only $(-2)$-curves
   as negative curves, see e.g. \cite{bf}. By a slight abuse of terminology we extend this definition to
   arbitrary surfaces and arbitrary negative curves.

   It is natural to compare the two chamber decompositions. 
Since the Zariski chambers
   need not in general be either open or closed, whereas Weyl chambers are by definition open, it is natural  to ask
\emph{under which condition is the interior of each Zariski chamber a simple Weyl chamber?}

   \noindent
   If it happens that all interiors of Zariski chambers coincide  with simple Weyl chambers,
   then we say that \emph{Zariski chambers are numerically
   determined} (by the sign of the intersection product with negative curves). A condition for Zariski chambers on K3 surfaces to be numerically determined was given in \cite[Theorem~1.2]{bf}.
   Here we prove the following criterion, that is valid for all smooth projective surfaces.
\begin{theorem}[A criterion for Zariski chambers to be numerically determined]\label{thm:main general}
   Let $X$ be a smooth projective surface. The following conditions are equivalent:
   \begin{itemize}
   \item[a)]  the interior of each Zariski chamber on $X$ is a simple Weyl chamber;
   \item[b)] if two  irreducible negative curves 
$C_1 \neq C_2$ on $X$ meet (i.e. $C_1.C_2 > 0$),  
then
   $$
C_1.C_2 \geq \sqrt{C_1^2 \cdot C_2^2} \, \, \, .
   $$
   \end{itemize}
\end{theorem}
\begin{remark}\label{rem:cond b) explained}
   In practice the condition b) in Theorem~\ref{thm:main general} means that the support of the (non-trivial) negative
   part of the Zariski decomposition of every big divisor on $X$ consists
   of pairwise disjoint 
curves.
Indeed, the condition in question implies that if the intersection matrix 
of two irreducible negative curves $C_1$, $C_2$ $\subset$ $X$ is negative-definite,
   then it is  diagonal.
\end{remark}
After proving  Theorem~\ref{thm:main general} in $\S$.\ref{sec-criterion},
 we study the relation between elliptic fibrations and Zariski chambers on Enriques surfaces in $\S$.\ref{enriques}.  It should be mentioned, that this note was motivated
 and inspired by the  earlier results of Bauer and Funke \cite{bf} on the K3 case.

\vspace*{1ex}
\noindent
{\em Convention:} In this note we work over the base field $\CC$. Elliptic fibrations are not assumed to have a section.
For basic facts on various types of divisors and cones associated to a smooth complex variety (resp. on elliptic fibrations) the reader should consult
\cite{PAG} (resp. \cite{shioda-schuett}).

\section{Proof of Theorem~\ref{thm:main general}} \label{sec-criterion}

\noindent
{\sl Proof of Theorem~\ref{thm:main general}.} The implication [a)$\Rightarrow$b)]: We argue by contraposition. Let $C_1$, $C_2$ be  negative curves on $X$ such that  $C_1.C_2 \neq 0$
 and the matrix $[C_i.C_j]_{i,j=1,2}$ is negative definite.
 To simplify our notation we put
   $$a:=-C_1^2,\;\; b:=-C_2^2,\;\; \mbox{ and }\; c:=C_1.C_2.$$
   Then we have $a, b, c >0$ and $ab>c^2$.

   We will construct two big divisors $D_1$ and $D_2$ such that
\begin{equation} \label{eq-npart}
\supp(N_{D_1})=\supp(N_{D_2})=\left\{C_1,C_2\right\},
\end{equation}
 the curves $C_1,C_2$ are the only irreducible curves on $X$ that meet the positive part $P_{D_1}$ (resp.  $P_{D_2}$) with multiplicity $0$,  and
\begin{equation} \label{eq-dweyl}
D_1.C_1<0,\; D_1.C_2<0\;\mbox{ but }\; D_2.C_1>0,\; D_2.C_2<0.
\end{equation}

   Let $H$ be an ample divisor on $X$. For $k \in \NN$ we define the following divisors
   $$T_k=(ab-c^2)H + k[(b(H.C_1)+c(H.C_2))C_1+(a(H.C_2)+c(H.C_1))C_2].$$
   Then, by direct computation, we have
   \begin{equation} \label{eq-tkc1c2}
    T_1.C_1=T_1.C_2=0 \quad  \mbox{ and }  \quad T_k.C_1 < 0,  T_k.C_2  < 0 \mbox{ for } k\geq 2 \, .
   \end{equation}
In particular $T_1$ is a nef divisor.  Moreover, by definition
\begin{equation} \label{eq-t1posint}
\mbox{for all irreducible  curves } C \subset X \mbox{ such that } C \neq C_1, C_2 \mbox{ we have } T_1.C > 0 \, .
\end{equation}

   Let $D_1=T_2$. Then
   $$D_1=T_1+[(b(H.C_1)+c(H.C_2))C_1+(a(H.C_2)+c(H.C_1))C_2]$$
   is the Zariski decomposition of $D_1$. Indeed, by \eqref{eq-t1posint} and \eqref{eq-tkc1c2} the divisor $T_1$ satisfies the conditions (Z1), (Z3).
   The choice of the curves $C_1$, $C_2$ implies that the condition (Z2) is satisfied.
   Since the Zariski decomposition of $D_1$ is uniquely determined by
(Z1)-(Z3),  the claim  follows.

   Finally, we define
   $$D_2:=(ab-c^2)H+(b(H.C_1)+c(H.C_2)+c)C_1+(a(H.C_2)+c(H.C_1)+2a)C_2.$$
   Then, by direct computation one gets  $D_2.C_1=ac>0$ and $D_2.C_2=c^2-2ab<0$.
   Moreover, the choice of the curves $C_1$, $C_2$ combined with  \eqref{eq-t1posint}, \eqref{eq-tkc1c2}
   implies that
   $$D_2=T_1+(cC_1+2aC_2)$$
   is the Zariski decomposition of $D_2$, so that $\supp(N_{D_2})=\left\{C_1,C_2\right\}$.

To complete the proof, observe that \eqref{eq-npart}
and  \eqref{eq-t1posint} yield that $D_1, D_2 \in \textstyle{\sum_{T_1}}$  (c.f. Definition~\ref{def-zchamber}).
On the other hand  \eqref{eq-dweyl} and \eqref{eq-t1posint}  yield that $D_1$ (resp. $D_2$)  belongs to a Weyl chamber (i.e. it does not belong to $C^{\perp}$
for an irreducible curve $C \subset X$). The two Weyl chambers in question do not coincide by \eqref{eq-dweyl}.

 \noindent
 The implication [b)$\Rightarrow$a)]:   The opposite implication is elementary.
 Let $D$ be a big divisor with Zariski decomposition   \eqref{eq-zariski} 
 that belongs to the interior of the chamber $\textstyle{\sum_{P}}$ for a big and nef $P$.

By \cite[Proposition~1.8]{bks} the only irreducible curves  $C \subset X$ such that
$C.P_D =0$ are the components of the negative part $N_D$. In particular, one has
$$
D.C \geq P_D.C > 0 \mbox{ for every irreducible curve } C \subset X \mbox{ such that } C \nsubseteq \mbox{supp}(N_D).
$$
Moreover, \cite[Proposition~1.8]{bks} yields that the components of $\mbox{supp}(N_D)$ are precisely the irreducible curves
that meet $P$ with multiplicity zero.

   Recall that the
   support of $N_D$ consists of mutually \emph{disjoint} curves $C_1,\ldots,C_s$ (see Remark~\ref{rem:cond b) explained}).
   Consequently, one obtains
   $$D.C_i=(P_D+N_D).C_i= \alpha_i C_i^2  < 0,$$
   which completes the proof. $\mbox{}$ \hfill $\Box$

As an immediate consequence one obtains the following corollary, that is a direct generalization of \cite[Theorem~1.2]{bf}.

\begin{cor} \label{cor-kodairazero}
Let $X$ be a smooth projective surface with $\mbox{kod}(X)=0$. Then the following conditions are equivalent:
\begin{itemize}
\item[a)] the simple Weyl chambers are the interiors of Zariski chambers on $X$,
\item[b)] there is no pair $C_1$, $C_2$ of smooth rational curves on  $X$ such that $C_1.C_2 = 1$.
\end{itemize}
\end{cor}
\begin{proof}  Recall that the only  irreducible curves on $X$ with negative self-intersection are $(-2)$-curves.
The latter are smooth and rational. Theorem~\ref{thm:main general} immediately yields the claim.
\end{proof}

We end this section with two examples: degree-$d$ hypersurfaces in $\PP^3$ for $d=3$ and $d \geq 4$.
\begin{example}
   Let $X_3$ be a smooth cubic surface in $\PP^3$. Obviously, the only negative curves on $X_3$
   are the $27$ lines $L$ and one has $L^2=(-1)$. Moreover, the intersection number of two lines never exceeds $1$. Thus Zariski chambers on $X_3$
   are numerically determined by  Theorem~\ref{thm:main general}.
   This follows also from \cite[Proposition 3.4]{bks}.
\end{example}
   The next example generalizes \cite[Proposition 3.1]{bf}.
\begin{example}\label{exp:deg d surface with lines}
   Let $X_d \subset \PP^3$ be a smooth surface of degree $d\geq 4$ which contains two
   intersecting lines $L_1, L_2$ (e.g. the degree-$d$ Fermat surface, see \cite{ssl}).
By adjunction,  we have
$$L_1^2 = L_2^2=(-d+2),$$
 so the condition b) in Theorem~\ref{thm:main general} is not
   satisfied for the intersection matrix of the lines in question. Hence,
   Zariski chambers on $X_d$ are not numerically determined.
\end{example}

\section{Zariski chambers on Enriques surfaces} \label{enriques}
   Let $X$ be an Enriques surface. Recall that for such surfaces the Weyl decomposition is given by irreducible $(-2)$-curves, i.e. {\sl simple roots}.
   A general Enriques surface carries no $(-2)$-curves, so the Zariski decomposition of $\mbox{Big}(X)$  becomes trivial. 
   An Enriques surfaces is called {\sl nodal} iff it contains a smooth rational curve. It is well-known that, if $\pi: Y \rightarrow X$ is the K3-cover of a  very general nodal Enriques surface $X$,
   then $Y$ is the resolution of a quartic symmetroid (see e.g. \cite{cossec-reye}). In particular,  we have $\rho(Y) = 11$, where
   $\rho(Y)$  stands for the Picard number of $Y$.

   The main aim of this section is to characterize  the nodal Enriques surfaces
   for which Zariski chambers are numerically determined in terms of elliptic fibrations.
   More precisely we will show the following theorem.
\begin{theo}\label{main}
   Let $X$ be an Enriques surface and let $\pi: Y \rightarrow X$ be its universal K3-cover. \\
   i) The following conditions are equivalent
   \begin{itemize}
      \item[a)] the simple Weyl chambers are the interiors of Zariski chambers on $X$,
      \item[b)] every fiber of every  elliptic fibration on $X$ has at most two components.
   \end{itemize}
\noindent
   ii) If none of the above conditions is satisfied, then we have  $\rho(Y) \geq 12$.
\end{theo}
\begin{proof}
 i) The implication [a)$\Rightarrow$b)]: Suppose there is an elliptic fibration on $X$ with a singular fiber that has  at least three components.
By the Kodaira classifiction of singular fibers (see e.g. \cite[$\S$~4]{shioda-schuett}), the fiber in question contains two smooth rational curves $C_1$, $C_2$
that meet transversally in exactly one point. Corollary~\ref{cor-kodairazero}  completes
the proof. \\
The implication [b)$\Rightarrow$a)]: Suppose that a  Weyl chamber on $X$ is not the interior of any   Zariski chamber. Using Corollary~\ref{cor-kodairazero} we find  two $(-2)$-curves $C_1$, $C_2$ on $X$ such that
$$C_1.C_2 = 1.$$
Let $M$ be the orthogonal complement of $\mbox{span}(C_1, C_2)$ in the
lattice $\mbox{Num}(X) = E_8 \oplus U$, where $U$ stands for the 
 hyperbolic plane.
By definition and Hodge Index Theorem, $M$ is  a rank-$8$ lattice of index $(1,7)$. In particular, the intersection form  is indefinite, so
we can apply  Meyer Theorem (see e.g. \cite[Corollary~2 on p.~43]{serre}) to find  a primitive class
\begin{equation} \label{eq-isotropv}
D \in M \mbox{ such that } D^2 = 0.
\end{equation}
By \cite[Proposition~16.1~(ii)]{bpv} we have either $|D|\neq \emptyset$ or  $|-D|\neq \emptyset$. Thus we can assume that
$|D|\neq \emptyset$. 

Recall that every  smooth rational curve $E$ on $X$ defines the Picard-Lefschetz reflection:
$$
{\mathfrak s}_{E} \, : \,  \mbox{H}^2(X,\ZZ) \ni D \mapsto D + (D.E) E \in \mbox{H}^2(X,\ZZ).
$$
Moreover, the counterimage of $E$ under $\pi$ decomposes into two disjoint smooth rational curves $E^{+}$, $E^{-}$.
Analogously, we have the  Picard-Lefschetz reflection  ${\mathfrak s}_{E^+}$  defined by $E^{+}$ on $\mbox{H}^2(Y,\ZZ)$ (see \cite[$\S$.~VIII.1]{bpv}).

By  \cite[Lemma~VIII.17.4]{bpv}, there exist smooth rational curves $E_1$, $\ldots$, $E_k$  on $X$ such that
for the composition of Picard-Lefschetz reflections ${\mathfrak p}_{X} :=  ({\mathfrak s}_{E_1} \circ    \ldots \circ {\mathfrak s}_{E_k}) $ we have
\begin{equation} \label{eq-halfpenc}
{\mathfrak p}_{X}(D) \mbox{ is a half-pencil of an elliptic fibration on } X \, .
\end{equation}
We put ${\mathfrak p}_{Y}  :=  ({\mathfrak s}_{E_1^{+}} \circ {\mathfrak s}_{E_1^{-}} \circ
\ldots \circ {\mathfrak s}_{E_k^{+}}  \circ {\mathfrak s}_{E_k^{-}})$.  As one can 
 check (see e.g. \cite[$\S$~2.3]{fourcusps}) we have
\begin{equation} \label{eq-plref}
{\mathfrak p}_{Y} \circ  \pi^{*} = \pi^{*} \circ   {\mathfrak p}_{X}  \mbox{ and } \pi_{*}({\mathfrak p}_{Y}(C_i^{+})) =  {\mathfrak p}_{X}(C_i)  \mbox{ for } i = 1, 2.
\end{equation}

To simplify our notation, we label the four curves $C_1^{\pm}$, $C_2^{\pm}$ on the K3 surface $Y$ in such  way   that $C_1^{+}.C_2^{+}= 1$. Since
$$
({\mathfrak p}_{Y}(C_i^{+}))^2 = -2, \mbox{ for } i= 1, 2, \, \, \,  \mbox{ and } \, \, \,
({\mathfrak p}_{Y}(C_1^{+}) + {\mathfrak p}_{Y}(C_2^{+}))^2 = -2,
$$
we can always assume that $|{\mathfrak p}_{Y}(C_1^{+})| \neq \emptyset$ and
$|{\mathfrak p}_{Y}(C_1^{+}+C_2^{+})| \neq \emptyset$. Recall that Picard-Lefschetz reflections are isometries. Thus,  from \eqref{eq-plref}  and \eqref{eq-halfpenc} we infer  that
${\mathfrak p}_{X}(C_1+C_2)$,  ${\mathfrak p}_{X}(C_1)$
are  effective divisors on $X$,  and their supports are  contained in a fiber of the elliptic fibration
given by $|2 {\mathfrak p}_{X}(D)|$.
Finally, the equality
$$
{\mathfrak p}_{X}(C_1).{\mathfrak p}_{X}(C_1 + C_2) = -1
$$
implies that the fiber in question is reducible, but it cannot be of the Kodaira type I$_2$.
The Kodaira classification of singular fibers  (see e.g. \cite[$\S$~4]{shioda-schuett}) completes the proof.

\noindent
ii) Let $\varphi: X \rightarrow \PP_1$ be an elliptic fibration on the Enriques surface $X$ and let $P_1, P_2$ be the images of half-pencils under $\varphi$.
Then, the K3-cover $Y$ is endowed with an  elliptic fibration  $\tilde{\varphi}$  such that the following diagram commutes:
\begin{equation} \label{firstdiagram}
\begin{xy}
\xymatrix{Y   \ar[r]^{2:1} \ar[d]_{\tilde{\varphi}} & X \ar[d]_{\varphi} \\
\PP_1  \ar[r]^{2:1}               & \PP_1
}
\end{xy}
\end{equation}
where the double cover $\PP_1 \rightarrow \PP_1$ is branched over the points $P_1, P_2$.
Thus \eqref{firstdiagram} induces the  commutative diagram
\begin{equation}
\begin{xy} \label{seconddiagram}
\xymatrix{\mbox{Jac}(Y)   \ar[r]^{2:1} \ar[d]_{\tilde{\psi}} & S \ar[d]_{\psi} \\
\PP_1  \ar[r]^{2:1}               & \PP_1   }
\end{xy}
\end{equation}
where $S := \mbox{Jac}(X)$ is a rational elliptic surface. In particular, we have $\rho(S) = \rho(X) = 10$, and
the elliptic fibrations $\varphi$ and $\psi$ have  singular fibers of the same Kodaira type.

Let $\mbox{T}_S$ be the trivial lattice of $S$ (see e.g. \cite[$\S$~6.4]{shioda-schuett}) and let $t := (\mbox{rank}(\mbox{T}_S) - 2)$.
Since we assumed a singular fiber of $\varphi$ to have at least $3$ components, we have  $t \geq 2$.

From the Shioda-Tate formula we get:
$$
10 = \rho(S) = 2 + t + \mbox{rank}(\mbox{MW}(S)),
$$
which yields $\mbox{rank}(\mbox{MW}(\mbox{Jac}(Y))) \geq \mbox{rank}(\mbox{MW}(S)) = (8 - t)$.

Assume that both half-pencils of the considered elliptic fibration are irreducible.
Then, each singular fiber of the fibration $\psi$ induces two singular fibers of the same type of the fibration $\tilde{\psi}$, so
for the trivial lattice of $\mbox{Jac}(Y)$ we get
$\mbox{rank}(\mbox{T}_{\operatorname{Jac}(Y)}) = 2 + 2t$.
Finally, we can apply the Shioda-Tate  formula on $\mbox{Jac}(Y)$  to get
\begin{equation} \label{eq-picjacy}
\rho(\operatorname{Jac}(Y)) \geq   10 + t  \geq 12 .
\end{equation}
If a half-pencil of $\varphi$ is reducible, then it is a fiber of the Kodaira type $I_k$. Thus it induces
an $I_{2k}$-fiber of  the fibration $\tilde{\psi}$ (see e.g. \cite[$\S.$~5.2]{shioda-schuett}). In particular, again we arrive at
\eqref{eq-picjacy}. \\
Finally, the equality $\rho(Y) = \rho(\operatorname{Jac}(Y))$ completes the proof.
\end{proof}

It should be emphasized, that the analogous statement does not  hold for elliptic K3 surfaces,  as the following example shows.
\begin{example}  \label{ex-K3}
(c.f. \cite[$\S$.~3]{bf}) Let $Y_4 \subset \PP_3(\CC)$ be a smooth quartic surface, such that

\begin{itemize}
\item[i)] a plane cuts the quartic $Y_4$ along a conic $C$ and two lines   $l', l''$,

\item[ii)] the Picard group $\mbox{Pic}(Y_4)$ is generated by $C$, $l'$, $l''$.
\end{itemize}
Obviously, the  line $l'$ (resp. $l''$) defines the elliptic fibration
$|{\mathcal O}_{Y_4}(1) - l'|$ (resp. $|{\mathcal O}_{Y_4}(1) - l''|$), but such a fibration has a unique reducible fiber and the latter is of the Kodaira type $I_2$.

Moreover, by \cite[Proposition~3.1~(ii)]{bf} the curves $C$, $l'$, $l''$ are the only $(-2)$-curves on $Y_4$. Since the conic $C$ meets each line with multiplicity two,
{\sl no fiber of an elliptic fibration on $Y_4$ has more than two components}.
On the other hand, the {\sl Zariski  and Weyl decompositions on $Y_4$ do not coincide} by \cite[Proposition~3.1~(iv)]{bf}.
\end{example}

\begin{remark}\label{rem:k3}
i) The K3 surface of Example~\ref{ex-K3} satisfies  the condition $\rho(Y_4) = 3$. An analysis of the proof  of  Theorem~\ref{main} (see \eqref{eq-isotropv}) shows that no
similar example with a K3 surface of Picard number $\geq 7$ can be constructed because one can use Meyer Theorem and Picard-Lefschetz reflections again. \\
ii) Obviously, given an elliptic K3 surface with a section and a reducible fiber of the elliptic fibration in question, \cite[Theorem~1.3]{bf} implies that Zariski chambers are not numerically determined.
\end{remark}

\vskip3mm
\noindent
{\bf Acknowledgement:} The first author would like to thank Th.~Bauer, M.~Joumaah and
M.~Sch\"utt for useful discussions.


\begin{thebibliography}{ACGH}

\bibitem{bauer}  Bauer, T.: \emph{A simple proof for the existence of Zariski decompositions on surfaces.} J. Algebraic Geom.~18 (2009), no. 4, 789-793.

\bibitem{bf}  Bauer, T., Funke, M.: \emph{Weyl and Zariski chambers on K3 surfaces.} Forum Math. 24 (2012), no. 3, 609-625.


\bibitem{bks}  Bauer, T., K\"uronya, A., Szemberg, T.:\emph{Zariski chambers, volumes, and stable base loci.} J. Reine Angew. Math.~576 (2004), 209-233.

\bibitem{bpv}  Barth, W., Hulek, K., Peters, C., van de Ven, A.:
\emph{Compact complex surfaces.} Second
edition, Erg. der Math. und ihrer Grenzgebiete, 3. Folge, Band 4. Springer, Berlin, 2004.

\bibitem{cossec-reye}  Cossec, F. R.: \emph{Reye congruences.} Trans. Amer. Math. Soc. 280 (1983), no. 2, 737--751.

\bibitem{cossec-dolgachev} Cossec, F. R., Dolgachev, I. V.: \emph{Enriques surfaces. I.} Progress in Math., vol.~76. Birkh\"auser, Basel,1989.

\bibitem{PAG}  Lazarsfeld, R.: \emph{Positivity in Algebraic Geometry~I, II}. Springer-Verlag, 2004.

\bibitem{fourcusps} Rams S., Sch\"utt M.: \emph{On Enriques surfaces with four cusps.} 2014, preprint arXiv:1404.3924.

\bibitem{shioda-schuett}  Sch\"utt, M., Shioda, T.:
\emph{Elliptic surfaces.},
Algebraic geometry in East Asia - Seoul 2008,
Advanced Studies in Pure Math.~60 (2010), 51-160.


\bibitem{ssl}  Sch\"utt, M., Shioda, T., van Luijk, R.: \emph{Lines on Fermat surfaces}. J. Number Theory 130 (2010), no. 9, 1939-1963.

\bibitem{serre} Serre, J.-P.: \emph{A course in arithmetic.} Graduate Texts in Mathematics, No. 7. Springer-Verlag, New York-Heidelberg, 1973.

\bibitem{zariski}  Zariski, O.: \emph{The theorem of Riemann-Roch for high multiples of an effective divisor on an
algebraic surface.} Ann. Math. 76, 560-615 (1962)


\end{thebibliography}
\end{document}